\documentclass[11pt]{amsart}
\calclayout

\usepackage{amssymb}
\usepackage{cite}

\newtheorem{thm}{Theorem}[section]

\newtheorem{prp}[thm]{Proposition}
\newtheorem{cor}[thm]{Corollary}

\theoremstyle{definition}

\theoremstyle{remark}
\newtheorem{rmk}[thm]{Remark}

\numberwithin{equation}{section}

\newcommand{\ov}[1]{\overline{#1}}

\begin{document} 

\title[Asymptotic boundary behavior of the Bergman curvatures]{Asymptotic boundary behavior of the Bergman curvatures of a pseudoconvex domain}

\author{Sungmin Yoo}
\address{Department of Mathematics, Pohang University of Science and Technology, Pohang, 790-784, Republic of Korea}
\email{sungmin@postech.ac.kr}
\thanks{This research was supported by the SRC-GAIA (NRF-2011-0030044) through the National Research Foundation of Korea (NRF) funded by the Ministry of Education.}



\begin{abstract}
We present a method of obtaining a lower bound estimate of the curvatures of the Bergman metric without using the regularity of the kernel function on the boundary. As an application, we prove the existence of an uniform lower bound of the bisectional curvatures of the Bergman metric of a smooth bounded pseudoconvex domain near the boundary with constant Levi rank.
\end{abstract}

\maketitle


\section{Introduction}

The curvatures of the Bergman metric, such as holomorphic sectional curvatures, bisectional curvatures and Ricci curvatures are important invariants in complex differential geometry. It is well-known that the bisectional curvatures (holomorphic sectional curvatures) and the Ricci curvatures of the Bergman metric of an $n$-dimensional complex manifold are always less than $2$ and $n+1$, respectively. Contrariwise, there exists an example constructed by Herbort \cite{Herbort07}, for which the holomorphic sectional curvature is not bounded from below in certain direction. To obtain a lower bound estimate of the curvatures of the K\"{a}hler metric from a direct calculation, estimates of derivatives of the potential function up to order 4 are necessary. Therefore, in the case of the Bergman metric, estimates of derivatives of the Bergman kernel function are needed.

For instance, Klembeck \cite{Klembeck78} showed that the holomorphic sectional curvatures of a $C^{\infty}$-smooth strongly pseudoconvex bounded domain in $\mathbb{C}^n$ approach $-4/(n+1)$ near the boundary by using Fefferman's asymptotic formula for the Bergman kernel function. Even though an asymptotic formula for the Bergman kernel function does not exist, McNeal proved in \cite{McNeal89} that the holomorphic sectional curvature of the Bergman metric for a pseudoconvex bounded domain in $\mathbb{C}^2$ with finite type boundary has to be bounded. He used the subelliptic estimate for the $\ov\partial$-Neumann problem which gives an information on the regularity of the kernel function up to the boundary.

However, Kim and Yu generalized in \cite{KY96} Klembeck's theorem to the case of $C^2$-smooth boundary, without using Fefferman's asymptotic formula. Therefore, it is natural to ask whether one can prove McNeal's result without using the results on the $\ov\partial$-Neumann problem. This was posed as Problem 26 in \cite{KSCV10}.

In this paper, we give an answer to this problem. The precise result is the following theorem.

\begin{thm}\label{main}
Let $\Omega\subset\subset\mathbb{C}^n$ be a pseudoconvex domain with smooth boundary. If the Levi form has constant rank $l$, $0\leq l\leq n-1$ on a neighborhood $V$ of $z_0\in\partial\Omega$, then there is a neighborhood $W$ of $z_0$ such that all the bisectional curvatures of the Bergman metric of $\Omega$ are bounded below by a negative constant in $W$. In particular, the holomorphic sectional curvatures and Ricci curvatures are also bounded below.
\end{thm}

Notice that if the Levi form is identically zero, then we cannot guarantee the regularity of the kernel function on the boundary since the pseudo-local property of Neumann operator does not hold (cf. \cite{Kerzman72}, \cite{Kohn72}).

The idea of the proof of Theorem \ref{main} is as follows: since the bisectional curvatures can be represented by Bergman's minimum integrals (Theorem \ref{bf}), we only need to get lower and upper bound estimates of the minimum integrals. Lower bounds are easily obtained by the minimum integrals of a subdomain. To achive an upper bound, we apply H\"{o}rmander's $L^2$-estimates of $\ov\partial$ with a plurisubharmonic function with large Hessians on the subdomain, which is constucted by Fu in \cite{Fu14}. The same argument with the result of Catlin \cite{Catlin89} gives an answer of Problem 26 in \cite{KSCV10}. Our result is an improvement upon the results of \cite{McNeal89} in this respect, since the holomorphic sectional curvatures does not control the bisectional curvatures and Ricci curvatures in general.

This paper is organized as follows: First, we briefly review fundamentals of Bergman geometry including Bergman's minimum integrals and applications of H\"{o}rmander's $L^2$-estimates of $\ov\partial$ to the minimum integrals. In section 4, we review the result of Fu \cite{Fu14} on the construction of a family of plurisubharmonic functions with large Hessians. Then we give a proof of Theorem \ref{main} with detail. In the last section, we show that our method is applicable to other pseudoconvex domains beyond the case of constant Levi rank (Theorem \ref{abstract}).\\


\textit{Acknowledgements}: The author would like to express his deep gratitude to Professor Kang-Tae Kim for valuable guidance and encouragements. This work is part of author's Ph.D. dissertation at Pohang University of Science and Technology.



\section{Preliminaries}

Let $\Omega$ be a bounded domain in $\mathbb{C}^n$. Define the {\it Bergman space} 
$$
\mathcal{A}^2(\Omega):=L^2(\Omega)\cap\mathcal{H}(\Omega),
$$
where $\mathcal{H}(\Omega)$ is the space of holomorphic functions on $\Omega$ and $L^2(\Omega)$ is the space of square integrable functions on $\Omega$. Let $\{\phi_j\}_{j=0}^{\infty}$ be a complete orthonormal basis for $\mathcal{A}^2(\Omega)$. Then the {\it Bergman kernel} and {\it Bergman metric} for $\Omega$ are defined by
$$
K_{\Omega}(z,\ov{z}):=\sum_{j=0}^{\infty}\phi_j(z)\ov{\phi_j(z)},
$$
$$
g_{\Omega}(z;X):=\sum_{j,k=1}^n\frac{\partial^2 \log K_{\Omega}(z,\ov{z})}{\partial z_j \partial \ov{z_k}}X_j\ov{X_k},
$$
where $z\in\Omega$ and $X=\sum_{i=1}^nX_i\frac{\partial}{\partial z_i}\in T_z^{1,0}(\Omega)$.\\

Since the Bergman metric is a K\"{a}hler metric, the {\it bisectional curvature} $B_{\Omega}(z;X,Y)$ at $z$ along the directions $X$ and $Y$ is defined by
$$
B_{\Omega}(z;X,Y)=\frac{ R_{\ov{h}jk\ov{l}} \ov{X_h}X_jY_k\ov{Y_l} }{g_{j\ov{k}} X_j\ov{X_k}g_{l\ov{m}}Y_l\ov{Y_m}},
$$
where
$$
R_{\ov{h}jk\ov{l}}=-\frac{\partial^2 g_{j\ov{h}}}{\partial z_k \partial \ov{z_l}}+g^{\nu\ov{\mu}}\frac{\partial g_{j\ov{\mu}}}{\partial z_k}\frac{\partial g_{\nu\ov{h}}}{\partial \ov{z_l}}.
$$
Here, we used the Einstein convention and $g^{\nu\ov{\mu}}$ denotes the components of the inverse matrix of $(g_{j\ov{k}})=(\frac{\partial^2 \log K_{\Omega}(z,\ov{z})}{\partial z_j \partial \ov{z_k}})$. The {\it holomorphic sectional curvature} $H_{\Omega}(z;X)$ and {\it Ricci curvature} $Ric_{\Omega}(z;X)$ at $z$ along the direction $X$ are given by
$$
H_{\Omega}(z;X)=B_{\Omega}(z;X,X),\ \ \ \ \ \ \ Ric_{\Omega}(z;X)=\sum_{j=1}^nB_{\Omega}(z;E^j,X),
$$
where $\{E^1,\ldots,E^n\}$ is a basis of $T_z^{1,0}(\Omega)$.\\

Given $p\in\Omega$ and $X=\sum_{i=1}^nX_i\frac{\partial}{\partial z_i}\big|_p, Y=\sum_{i=1}^nY_i\frac{\partial}{\partial z_i}\big|_p\in T_p^{1,0}(\Omega)$, consider the minimum integrals:\medskip
$$
I^0_{\Omega}(p):=\inf_{f\in \mathcal{A}^2(\Omega)}\{\|f\|_{L^2(\Omega)}^2: f(p)=1\},
$$
$$
I^1_{\Omega}(p;X):=\inf_{f\in \mathcal{A}^2(\Omega)}\{\|f\|_{L^2(\Omega)}^2: f(p)=0, \sum_{i=1}^nX_i\frac{\partial f}{\partial z_i}(p)=1\},
$$
$$
I^2_{\Omega}(p;X,Y):=\inf_{f\in \mathcal{A}^2(\Omega)}\{\|f\|_{L^2(\Omega)}^2: f(p)=0, df(p)=0, \sum_{j,k=1}^nX_jY_k\frac{\partial^2 f}{\partial z_j \partial z_k}(p)=1\}.
$$

It is easy to see from the definitions that if $p\in\Omega'\subset\Omega$, then
$$
I^0_{\Omega'}(p)\leq I^0_{\Omega}(p),\ \ \  I^1_{\Omega'}(p;X)\leq I^1_{\Omega}(p;X),\ \ \  I^2_{\Omega'}(p;X,Y)\leq I^2_{\Omega}(p;X,Y).
$$
If $f:\Omega_1\rightarrow\Omega_2$ is a biholomorphism, then the following transformation formula hold.
\begin{align*}
I^0_{\Omega_1}(p)|\det J_{\mathbb{C}}f(p)|^2&=I^0_{\Omega_2}(f(p)),\\
I^1_{\Omega_1}(p;X)|\det J_{\mathbb{C}}f(p)|^2&=I^1_{\Omega_2}(f(p);df(X)),\\
I^2_{\Omega_1}(p;X,Y)|\det J_{\mathbb{C}}f(p)|^2&=I^2_{\Omega_2}(f(p);df(X),df(Y)),
\end{align*}
where $J_{\mathbb{C}}f(p)$ denotes the complex Jacobian matrix of $f$ at $p$. Moreover, the following formula were proved in the 1930's by Bergman and Fuchs (see \cite{Bergman70}).

\begin{thm}[Bergman-Fuchs]\label{bf}
$$
K_{\Omega}(p,\ov{p})=\frac{1}{I^0_{\Omega}(p)},\ \  \ \ \ \  g_{\Omega}(p;X)=\frac{I^0_{\Omega}(p)}{I^1_{\Omega}(p;X)},
$$
$$
H_{\Omega}(p;X)=2-\frac{(I^1_{\Omega}(p;X))^2}{I^0_{\Omega}(p)I^2_{\Omega}(p;X,X)}.
$$
\end{thm}

Later, Pagano observed that the bisectional curvature satisfies the following polarized identity (cf. \cite{KK03}).
$$
B_{\Omega}(p;X,Y)=2-\frac{I^1_{\Omega}(p;X)I^1_{\Omega}(p;Y)}{I^0_{\Omega}(p)I^2_{\Omega}(p;X,Y)}.
$$


\section{H\"{o}rmander's $L^2$ estimates and applications}

In this section, we consider two applications of the following theorem:

\begin{thm}[H\"{o}rmander \cite{Hormander65}]\label{hormander}
Let $\Omega$ be a bounded pseudoconvex domain in $\mathbb{C}^n$, let $\phi$ be a plurisubharmonic function in $\Omega$ and $c\in C(\Omega)$ be a positive function such that
$$
\sum^n_{j,k=1}\frac{\partial^2\phi(z)}{\partial z_j\partial \ov{z_k}}X_j\ov{X_k}\geq c\sum^n_{j=1}|X_j|^2,
$$
for all $p\in\Omega$ and $X\in\mathbb{C}^n$. For every $f\in L^2_{(0,1)}(\Omega,\phi)$ satisfying $\ov\partial f=0$, there is a function $u\in L^2(\Omega,\phi)$ such that $\ov\partial u=f$ and
$$
\int_{\Omega}|u|^2e^{-\phi}\leq\int_{\Omega}\frac{|f|^2}{c}e^{-\phi},
$$
provided that the right hand side is finite.
\end{thm}

As the first application of Theorem \ref{hormander}, we introduce the following theorem on localization of the minimum integrals. The inequalities in (1) and (2) were proved in \cite{Ohsawa80} and \cite{DFH84}. Although the proof of (3) is almost identical, we prove this for the convenience.

\begin{thm}\label{local}
Let $\Omega$ be a bounded pseudoconvex domain in $\mathbb{C}^n$ and $z_0\in\partial\Omega$. Suppose $W\subset\subset U$ are small open neighborhoods of $z_0$. Then there exist positive constants $C_0,C_1,$ and $C_2$ such that
\begin{enumerate}
	\item $I^0_{U\cap\Omega}(p)\leq I^0_{\Omega}(p)\leq C_0I^0_{U\cap\Omega}(p)$,\medskip
	\item $I^1_{U\cap\Omega}(p;X)\leq I^1_{\Omega}(p;X)\leq C_1I^1_{U\cap\Omega}(p;X)$,\medskip
	\item $I^2_{U\cap\Omega}(p;X,Y)\leq I^2_{\Omega}(p;X,Y)\leq C_2I^2_{U\cap\Omega}(p;X,Y)$,
\end{enumerate}
for all $p\in W\cap\Omega$ and $X\in\mathbb{C}^n$.
\end{thm}

\begin{proof}
Let $f$ be a holomorphic function on $U\cap\Omega$ such that 
$$
f(p)=0,\ \ df(p)=0,\ \  \sum_{j,k=1}^nX_jY_k\frac{\partial^2 f}{\partial z_j \partial z_k}(p)=1, {\ \ \rm and\ \ } \|f\|^2_{L^2(U\cap\Omega)}=I^2_{U\cap\Omega}(p;X,Y).
$$ 
Fix another open set $U_0$, $W\subset\subset U_0\subset\subset U$ and a cut-off function $\chi\in C_0^{\infty}(U), 0\leq\chi\leq 1, \chi|_{U_0}\equiv 1$. Put $\phi(z)=2(n+2)\log|z-p|+\log(1+|z|^2)$, and set $\alpha=\ov\partial(\chi f)=f\ov\partial\chi$. Applying Theorem \ref{hormander}, one gets a solution $u$ to $\ov\partial u=\alpha$ such that
$$
\int_{\Omega}\frac{|u|^2}{|z-p|^{2n+4}(1+|z|^2)^2}\leq\int_{\Omega\cap(U-U_0)}\frac{|\alpha|^2}{|z-p|^{2n+4}}.
$$
Since the right hand side is bounded, $u(p)=du(p)=\frac{\partial^2 u}{\partial z_j \partial z_k}(p)=0$. Moreover, the above inequality implies that $\|u\|^2_{L^2(\Omega)}\lesssim\|f\|^2_{L^2(U\cap\Omega)}.$
Now define a holomorphic function $F=\chi f-u$ on $\Omega$. Then $\|F\|^2_{L^2(\Omega)}\lesssim \|f\|^2_{L^2(U\cap\Omega)}$, and
$$
F(p)=f(p)=0,\ \ dF(p)=df(p)=0, {\ \ \rm and\ \ } \sum_{j,k=1}^nX_jY_k\frac{\partial^2 F}{\partial z_j \partial z_k}(p)=1.
$$
This concludes $I^2_{\Omega}(p;X,Y)\lesssim I^2_{U\cap\Omega}(p;X,Y)$.
\end{proof}

\begin{rmk}
In \cite{Hormander65} and \cite{KY96}, on the additional assumption that $z_0$ admits a holomorphic local peak function, they proved the following:
$$
\lim_{U\cap\Omega\ni z\rightarrow z_0}\frac{I^j_{U\cap\Omega}(z;X,Y)}{I^j_{\Omega}(z;X,Y)}=1,
$$
for $j=0,1,2$ and for any neighborhood $U$ of $z_0$.
\end{rmk}

\begin{rmk}
The right side inequality in (3) is not necessary in the proof of Theorem \ref{main}. However, we include it because of its possible usefulness for an upper bound estimates of the curvatures of the Bergman metric.
\end{rmk}

Catlin proved the following theorem by a slight modification of the proof of Theorem \ref{hormander}.

\begin{thm}[Catlin \cite{Catlin89}]\label{catlin}
Let $\Omega$ be a bounded pseudoconvex domain with smooth boundary in $\mathbb{C}^n$ and $p\in\Omega$. Let $\beta_1,\ldots,\beta_n$ be given positive numbers. Assume that there exists a function $\phi\in C^3(\ov{\Omega})$ and positive constants $M,\widetilde{C},C_{\alpha}$ such that
\begin{enumerate}
	\item $|\phi(z)|\leq M$, for $z\in\Omega$;
	\item $\phi$ is plurisubharmonic in $\Omega$;
	\item $P:=P(p,\beta)=\{z\in\mathbb{C}^n; |z_j-p_j|<\beta_j, 1\leq j\leq n\}\subset\Omega$;
	\item for all $z\in P$ and $X\in\mathbb{C}^n$,
	$$
	\sum^n_{j,k=1}\frac{\partial^2\phi(z)}{\partial z_j\partial \ov{z_k}}X_j\ov{X_k}\geq \widetilde{C} \sum^n_{j=1}\frac{|X_j|^2}{\beta^2_j};
	$$
	\item for all $\alpha=(\alpha_1,\ldots,\alpha_n)$ with $|\alpha|=\sum_j\alpha_j\leq 3$ and $z\in P$,
	$$
	|D^{\alpha}\phi(z)|\leq C_{\alpha} \prod^n_{j=1} \beta_j^{-\alpha_j},
	$$
	where $D^{\alpha}=D_1^{\alpha_1}\cdots D_n^{\alpha_n}$.
\end{enumerate}
Then there exists a positive constant $C$, only depends on the constants $M,\widetilde{C},C_{\alpha}$, such that
$$
C^{-1}\prod^n_{j=1}\frac{1}{\beta_j^2}\leq K_{\Omega}(p,\ov{p})\leq C\prod^n_{j=1}\frac{1}{\beta_j^2},
$$
and
$$
C^{-1}\sum^n_{j=1}\frac{|X_j|^2}{\beta_j^2}\leq g_{\Omega}(p;X)\leq C\sum^n_{j=1}\frac{|X_j|^2}{\beta_j^2}.
$$
\end{thm}


\section{Plurisubharmonic functions on centered domains}

In this section, we recall results of Fu in \cite{Fu14}. From now on, suppose that $\Omega\subset\subset\mathbb{C}^n$ is a smooth pseudoconvex domain with constant Levi rank near $z_0\in\partial\Omega$. More precisely, there exists a neighborhood $V$ of $z_0$ such that the number of non-zero eigenvalues of the Levi form at $q$ is $l, 0\leq l\leq n-1$, for all $q\in V\cap\partial\Omega$.

\begin{prp}[Fu \cite{Fu14}]\label{center}
There exists a neighborhood $U\subset\subset V$ of $z_0$ such that for each $q\in U\cap\partial\Omega$, there is a biholomorphic mapping $\zeta=\Phi_q(z)$ from $U$ onto the unit ball $B(0,1)$ that satisfies
\begin{enumerate}
	\item $\Phi_q(q)=0$;
	\item $\Phi_q$ depends smoothly on $q$;
	\item $\Phi_q(U\cap\partial\Omega)$ has a defining function of the form
	$$
	\rho(\zeta)={\rm Re}\zeta_1+\sum_{j=2}^{l+1}\lambda_j|\zeta_j|^2+O(|\zeta'|^2\cdot|\zeta''|+|\zeta'|^3+|{\rm Im}\zeta_1|\cdot|\zeta|)
	$$
	near $0$, where $\lambda_j, 2\leq j\leq l+1$, are positive constants depending smoothly on $q$.
\end{enumerate}
\end{prp}

We will call the above biholomorphic mapping $\Phi_q$ as the {\it centering map} at $q$, since it is a variation of the scaling map of Pinchuk's scaling procedure. For each {\it centered domain} $\widetilde{\Omega}_q$, Fu constructed the following bounded plurisubharmonic function with large Hessian:

\begin{thm}[Fu \cite{Fu14}]\label{psh}
Let $W\subset\subset U$ be a neighborhood of $z_0$. Then for any $q\in W\cap\partial\Omega$ and any sufficiently small $\delta>0$, there exists a function $g_{q,\delta}\in C^{\infty}(\ov{\widetilde{\Omega}_q})$ and constants $a,C>0$ and $C_{\alpha}$, independent of $q$ and $\delta$, such that
\begin{enumerate}
	\item $|g_{q,\delta}(\zeta)|\leq 1, z\in\widetilde{\Omega}_q$;
	\item $g_{q,\delta}$ is a plurisubharmonic on $\widetilde{\Omega}_q$;
	\item for $\zeta\in P_{\delta,a}\cap\widetilde{\Omega}_q$ and $\xi\in\mathbb{C}^n$,
	$$
	\sum^n_{j,k=1}\frac{\partial^2g_{q,\delta}(\zeta)}{\partial \zeta_j\partial \ov{\zeta_k}}\xi_j\ov{\xi_k} \geq \frac{1}{C} \Big(\frac{|\langle\partial\rho(\zeta),\xi\rangle|^2}{\delta^2}+\sum^{n-l}_{j=2}\frac{|\xi_j|^2}{\delta}+\sum^{n}_{k=n-l+1}|\xi_k|^2\Big);
	$$
	where $P_{\delta,a}:=\{\zeta\in\mathbb{C}^n; |\zeta_1|<a\delta, |\zeta_j|<a\delta^{\frac{1}{2}}, |\zeta_k|<a \}$, $2\leq j\leq n-l$, $n-l+1\leq k\leq n$.
	\item for $\zeta\in P_{\delta,a}\cap\widetilde{\Omega}_q$,
	$$
	|D^{\alpha}g_{q,\delta}(\zeta)|\leq C_{\alpha}\delta^{-(\alpha_1+\frac{1}{2}\sum_{j=2}^{n-l}\alpha_j)},
	$$
	where $D^{\alpha}=D_1^{\alpha_1}\cdots D_n^{\alpha_n}$.
\end{enumerate}
\end{thm}


\section{Proof of the Main Theorem}

We prove Theorem \ref{main} in this section. Theorem \ref{bf} implies that it suffices to show that there is a neighborhood $W$ of $z_0$ and a constant $C>0$, independent of $p$ and $X$ such that
$$
\frac{I^1_{\Omega}(p;X)I^1_{\Omega}(p;Y)}{I^0_{\Omega}(p)I^2_{\Omega}(p;X,Y)}\leq C,
$$
for all $p\in W\cap\Omega$ and $X,Y\in\mathbb{C}^n$.

Let $W\subset\subset U$ be neighborhoods of $z_0$ in Proposition \ref{center} and Theorem \ref{psh}. From the localization of minimum integrals (Theorem \ref{local}), we only need to show that there is constant $C>0$, independent of $p$ and $X$ such that
$$
\frac{I^1_{U\cap\Omega}(p;X)I^1_{U\cap\Omega}(p;Y)}{I^0_{U\cap\Omega}(p)I^2_{U\cap\Omega}(p;X,Y)}\leq C,
$$
for all $p\in W\cap\Omega$ and $X,Y\in\mathbb{C}^n$.

Denote by $\pi(p)$ the projection of $p$ onto the boundary $\partial\Omega$. It is well-defined near the boundary. We may assume that for all $p\in W\cap\Omega$, $\pi(p)\in U\cap\partial\Omega$. Then Proposition \ref{center} implies that for each $\pi(p)\in U\cap\partial\Omega$, there exists the centering map given by $\zeta=\Phi_{\pi(p)}(z)$ such that
$$
\zeta_p:=\Phi_{\pi(p)}(p)=(-\delta(p),0,\ldots,0),
$$ 
for some $\delta(p)>0$. By the transformation formula for minimum integrals, the above inequality is equivalent to
$$
\frac{I^1_{\widetilde{\Omega}_{\pi(p)}}(\zeta_p;\xi)I^1_{\widetilde{\Omega}_{\pi(p)}}(\zeta_p;\eta)}{I^0_{\widetilde{\Omega}_{\pi(p)}}(\zeta_p)I^2_{\widetilde{\Omega}_{\pi(p)}}(\zeta_p;\xi,\eta)}\leq C,
$$
where $\widetilde{\Omega}_{\pi(p)}=\Phi_{\pi(p)}(U\cap\Omega)$ and $\xi=d\Phi_{\pi(p)}(X),\eta=d\Phi_{\pi(p)}(Y)$.

Then Theorem \ref{psh} implies that there exist a constant $a>0$ and a family of bounded plurisubharmonic functions $\{\phi_p\}_{p\in W\cap\Omega}$ defined by
$$
\phi_p(\zeta):=g_{\pi(p),\delta(p)}(\zeta)
$$ 
such that $\phi_p$ have sufficiently large Hessians on $P_{\delta(p),a}\cap\widetilde{\Omega}_{\pi(p)}$.

For $c>0$, let
$$
Q_{\delta(p),c}:=\{\zeta\in\mathbb{C}^n; |\zeta_1+\delta(p)|<c\delta(p), |\zeta_j|<c\delta(p)^{\frac{1}{2}}, |\zeta_k|<c \},
$$
where $2\leq j\leq n-l$ and $n-l+1\leq k\leq n$.

Proposition \ref{center} implies that there exists sufficiently small constant $c>0$ such that
$$
Q_{\delta(p),c}\subset P_{\delta(p),a}\cap\widetilde{\Omega}_{\pi(p)},
$$
and
$$
\sum^n_{j,k=1}\frac{\partial^2\phi_p(\zeta)}{\partial \zeta_j\partial \ov{\zeta_k}}\xi_j\ov{\xi_k} \gtrsim \frac{|\xi_1|^2}{\delta(p)^2}+\sum^{n-l}_{j=2}\frac{|\xi_j|^2}{\delta(p)}+\sum^{n}_{k=n-l+1}|\xi_k|^2,
$$
for all $\zeta\in Q_{\delta(p),c}$ and $\xi\in\mathbb{C}^n$ (see \cite{Fu14}).

Since the Bergman kerenl and the Bergman metric at $p$ of the polydisc $P$ are given by
$$
K_P(p,\ov p)=\pi^{-n}\prod^n_{j=1}\frac{1}{\beta_j^2}, {\text\ \ \ \ \ \ } g_P(p;X)=2\sum^n_{j=1}\frac{|X_j|^2}{\beta_j^2},
$$
Theorem \ref{catlin} and Theorem \ref{bf} imply that 
$$
I^0_{\Omega}(p)\approx I^0_{P}(p) {\text\ \ \  and\ \ \ } I^1_{\Omega}(p;X)\approx I^1_{P}(p;X).
$$
Applying Theorem \ref{catlin} with $\widetilde{\Omega}_{\pi(p)}$, $\phi_p$, and the polydiscs $Q_{\delta(p),c}$ centered at $\zeta_p$, we get
$$
\frac{I^1_{\widetilde{\Omega}_{\pi(p)}}(\zeta_p;\xi)I^1_{\widetilde{\Omega}_{\pi(p)}}(\zeta_p;\eta)}{I^0_{\widetilde{\Omega}_{\pi(p)}}(\zeta_p)I^2_{\widetilde{\Omega}_{\pi(p)}}(\zeta_p;\xi,\eta)}\lesssim \frac{I^1_{Q_{\delta(p),c}}(\zeta_p;\xi)I^1_{Q_{\delta(p),c}}(\zeta_p;\eta)}{I^0_{Q_{\delta(p),c}}(\zeta_p)I^2_{Q_{\delta(p),c}}(\zeta_p;\xi,\eta)}=2-B_{Q_{\delta(p),c}}(\zeta_p;\xi,\eta).
$$

Therefore, the proof is completed by the following uniform estimate of the bisectional curvatures of polydisc at the center:
$$
0\leq-B_P(p;X,Y)=\frac{\sum^n_{j=1}\frac{|X_j|^2|Y_j|^2}{\beta_j^4}}{(\sum^n_{j=1}\frac{|X_j|^2}{\beta_j^2})(\sum^n_{k=1}\frac{|Y_k|^2}{\beta_k^2})}\leq 1.
$$


\section{Beyond the case of constant Levi rank}

Our method can be also applied to not only the case of constant Levi rank but also other bounded pseudoconvex domains, such as finite type domains in $\mathbb{C}^2$ and convex domains with finite type in $\mathbb{C}^n$. Since the proof is almost the same, we omit the proof of the following:

\begin{thm}\label{abstract}
Let $\Omega$ be a smooth pseudoconvex bounded domain in $\mathbb{C}^n$, and $z_0\in\partial\Omega$. Let $V$ be a neighborhood of $z_0$. Suppose that for each $p\in V\cap\Omega$, there exist a biholomorphic map $\zeta=\Psi_p(z)$ on $V\cap\Omega$, and a polydisc $P_p$ centered at $\Psi_p(p)$ with radius $\beta(p)=(\beta_1(p),\ldots,\beta_n(p))$ such that
$$
P_p\subset\Omega_p:=\Psi_p(V\cap\Omega).
$$
Assume that for each $p\in V\cap\Omega$, there exists a function $\phi_p\in C^3(\ov{\Omega_p})$ and positive constants $M,\widetilde{C},C_{\alpha}$, independent of $p$ such that
\begin{enumerate}
	\item $|\phi_p(\zeta)|\leq M$, for $\zeta\in\Omega_p$;
	\item $\phi_p$ is plurisubharmonic in $\Omega_p$;
	\item for all $\zeta\in P_p$ and $X\in\mathbb{C}^n$,
	$$
	\sum^n_{j,k=1}\frac{\partial^2\phi_p(\zeta)}{\partial \zeta_j\partial \ov{\zeta_k}}X_j\ov{X_k}\geq \widetilde{C} \sum^n_{j=1}\frac{|X_j|^2}{\beta_j(p)^2};
	$$
	\item for all $\alpha=(\alpha_1,\ldots,\alpha_n)$ with $|\alpha|=\sum_j\alpha_j\leq 3$ and $\zeta\in P_p$,
	$$
	|D^{\alpha}\phi_p(\zeta)|\leq C_{\alpha} \prod^n_{j=1} \beta_j(p)^{-\alpha_j},
	$$
	where $D^{\alpha}=D_1^{\alpha_1}\cdots D_n^{\alpha_n}$.
\end{enumerate}
Then there is a neighborhood $U$ of $z_0$ such that all the bisectional curvatures of the Bergman metric of $\Omega$ are bounded below by a negative constant in $U$. In particular, the holomorphic sectional curvatures and Ricci curvatures are also bounded below.
\end{thm}

Proposition \ref{center} and Theorem \ref{psh} imply that Theorem \ref{main} is a corollary of Theorem \ref{abstract}. Moreover, we have following two corollaries.

\begin{cor}
Let $\Omega$ be a smooth pseudoconvex bounded domain in $\mathbb{C}^2$, and $z_0\in\partial\Omega$ be a point of finite type. Then there is a neighborhood $U$ of $z_0$ such that all the bisectional curvatures of the Bergman metric of $\Omega$ are bounded below by a negative constant in $U$.
\end{cor}
\begin{proof}
Use Proposition 1.1 and Proposition 4.2 in \cite{Catlin89}.
\end{proof}

\begin{cor}
Let $\Omega$ be a smooth pseudoconvex bounded domain in $\mathbb{C}^n$, and $z_0\in\partial\Omega$ be a point of finite type in the sense of D'Angelo. Let $V$ be a neighborhood of $z_0$ such that $V\cap\partial\Omega$ is convex. Then there is a neighborhood $U$ of $z_0$ such that all the bisectional curvatures of the Bergman metric of $\Omega$ are bounded below by a negative constant in $U$.
\end{cor}
\begin{proof}
Use Proposition 2.2 and 3.1 in \cite{Chen89} (cf. Proposition 2.1 and 3.1 in \cite{McNeal94}).
\end{proof}


\bibliographystyle{amsplain}


\end{document}